\documentclass[12pt,a4paper]{amsart}
\usepackage{graphicx}
\usepackage{yhmath}
\usepackage{amsmath,verbatim}
\usepackage{amssymb}
\usepackage{pgf,tikz}
\usetikzlibrary{arrows}
\usepackage[utf8]{inputenc}

\theoremstyle{plain}
\newtheorem{thm}{Theorem}[section]
\newtheorem{lemma}[thm]{Lemma}
\newtheorem{corollary}[thm]{Corollary}
\newtheorem{prop}[thm]{Proposition}
\newtoks\prt

\theoremstyle{definition}

\newtheorem{remark}[thm]{Remark}

\newtheorem{definition}[thm]{Definition}

\def\eqn#1$$#2$${\begin{equation}\label#1#2
\end{equation}}
\newcommand{\abs}[1]{\lvert#1\rvert}
\newcommand{\norm}[1]{\lVert#1\rVert}
\numberwithin{equation}{section}

\def\diam{\operatorname{diam}}
\def\dist{\operatorname{dist}}

\def\cR{\mathcal{R}}

\def\G{\mathcal{G}}
\def\H{\mathcal{H}}
\def\phi{\varphi}
\def\epsilon{\varepsilon}
\def\eps{\varepsilon}
\def\en{\mathbb N}

\def\er{\mathbb R}
\def\R{\mathbb R}

\newtoks\by
\newtoks\paper
\newtoks\book
\newtoks\jour
\newtoks\yr
\newtoks\pages
\newtoks\vol
\newtoks\publ
\def\ota{{\hbox\vol{???}}}
\def\cLear{\by=\ota\paper=\ota\book=\ota\jour=\ota\yr=\ota
\pages=\ota\vol=\ota\publ=\ota}
\def\endpaper{\the\by, {\the\paper},
\textit{\the\jour} \textbf{\the\vol} (\the\yr), \the\pages.\cLear}
\def\endbook{\the\by, \textit{\the\book}, \the\publ.\cLear}
\def\endprep{\the\by, \textit{\the\paper}, \the\jour.\cLear}
\def\endyearprep{\the\by, \textit{\the\paper}, \the\jour, (\the\yr).\cLear}
\def\name#1#2{#2 #1}

\def\de0#1{\rule[3pt]{#1}{0.4pt} \hspace{-0.1pt} \rule[3.05pt]{0.05pt}{0.4pt} \hspace{-0.1pt} \rule[3.1pt]{0.05pt}{0.4pt} \hspace{-0.1pt} \rule[3.15pt]{0.05pt}{0.4pt} \hspace{-0.1pt} \rule[3.2pt]{0.05pt}{0.4pt} \hspace{-0.1pt} \rule[3.25pt]{0.05pt}{0.4pt} \hspace{-0.1pt} \rule[3.3pt]{0.05pt}{0.4pt} \hspace{-0.1pt} \rule[3.35pt]{0.05pt}{0.4pt} \hspace{-0.1pt} \rule[3.4pt]{0.05pt}{0.4pt} \hspace{-0.1pt} \rule[3.45pt]{0.05pt}{0.4pt} \hspace{-0.1pt} \rule[3.5pt]{0.05pt}{0.4pt} \hspace{-0.1pt} \rule[3.55pt]{0.05pt}{0.4pt} \hspace{-0.1pt} \rule[3.6pt]{0.05pt}{0.4pt} \hspace{-0.1pt} \rule[3.65pt]{0.05pt}{0.4pt} \hspace{-0.1pt} \rule[3.7pt]{0.05pt}{0.4pt} \hspace{-0.1pt} \rule[3.75pt]{0.05pt}{0.4pt} \hspace{-0.1pt} \rule[3.8pt]{0.05pt}{0.4pt} \hspace{-0.1pt} \rule[3.85pt]{0.05pt}{0.4pt} \hspace{-0.1pt} \rule[3.9pt]{0.05pt}{0.4pt} \hspace{-0.1pt} \rule[3.95pt]{0.05pt}{0.4pt} \hspace{-0.1pt} \rule[4.0pt]{0.05pt}{0.4pt} \hspace{-0.1pt} \rule[4.05pt]{0.05pt}{0.4pt} \hspace{-0.1pt} \rule[4.1pt]{0.05pt}{0.4pt} \hspace{-0.1pt} \rule[4.15pt]{0.05pt}{0.4pt} \hspace{-0.1pt} \rule[4.2pt]{0.05pt}{0.4pt} \hspace{-0.1pt} \rule[4.25pt]{0.05pt}{0.4pt} \hspace{-0.1pt} \rule[4.3pt]{0.05pt}{0.4pt} \hspace{-0.1pt} \rule[4.35pt]{0.05pt}{0.4pt} \hspace{-0.1pt} \rule[4.4pt]{0.05pt}{0.4pt} \hspace{-0.1pt} \rule[4.45pt]{0.05pt}{0.4pt} \hspace{-0.1pt} \rule[4.5pt]{0.05pt}{0.4pt} \hspace{-0.1pt} \rule[4.55pt]{0.05pt}{0.4pt} \hspace{-0.1pt} \rule[4.6pt]{0.05pt}{0.4pt} \hspace{-0.1pt} \rule[4.65pt]{0.05pt}{0.4pt} \hspace{-0.1pt} \rule[4.7pt]{0.05pt}{0.4pt} \hspace{-0.1pt} \rule[4.75pt]{0.05pt}{0.4pt} \hspace{-0.1pt} \rule[4.8pt]{0.05pt}{0.4pt} \hspace{-0.1pt} \rule[4.85pt]{0.05pt}{0.4pt} \hspace{-0.1pt} \rule[4.9pt]{0.05pt}{0.4pt} \hspace{-0.1pt} \rule[4.95pt]{0.05pt}{0.4pt} \hspace{-0.1pt} \rule[5.0pt]{0.05pt}{0.4pt} \hspace{-0.1pt} \rule[5.05pt]{0.05pt}{0.4pt} \hspace{-0.1pt} \rule[5.1pt]{0.05pt}{0.4pt} \hspace{-0.1pt} \rule[5.15pt]{0.05pt}{0.4pt} \hspace{-0.1pt} \rule[5.2pt]{0.05pt}{0.4pt} \hspace{-0.1pt} \rule[5.25pt]{0.05pt}{0.4pt} \hspace{-0.1pt} \rule[5.3pt]{0.05pt}{0.4pt} \hspace{-0.1pt} \rule[5.35pt]{0.05pt}{0.4pt} \hspace{-0.1pt} \rule[5.4pt]{0.05pt}{0.4pt} \hspace{-0.1pt} \rule[5.45pt]{0.05pt}{0.4pt} \hspace{-0.1pt} \rule[5.5pt]{0.05pt}{0.4pt} \hspace{-0.1pt} \rule[5.55pt]{0.05pt}{0.4pt} \hspace{-0.1pt} \rule[5.6pt]{0.05pt}{0.4pt} \hspace{-0.1pt} \rule[5.65pt]{0.05pt}{0.4pt} \hspace{-0.1pt} \rule[5.7pt]{0.05pt}{0.4pt} \hspace{-0.1pt} \rule[5.75pt]{0.05pt}{0.4pt} \hspace{-0.1pt} \rule[5.8pt]{0.05pt}{0.4pt} \hspace{-0.1pt} \rule[5.85pt]{0.05pt}{0.4pt} \hspace{-0.1pt} \rule[5.9pt]{0.05pt}{0.4pt} \hspace{-0.1pt} \rule[5.95pt]{0.05pt}{0.4pt} \hspace{-0.1pt} \rule[6.0pt]{0.05pt}{0.4pt}}	

\MakeRobust{\ref}

\makeatletter
\newcommand{\labeltext}[2]{%
	\@bsphack
	\def\@currentlabel{#1}{\label{#2}}%
	\@esphack
}
\makeatother

\def\step#1#2#3{\par \noindent{{\\ \bf Step~\labeltext{#1}{#3}#1. }{\bf #2. }}}

\begin{document}

\title[Classification of strict limits of planar BV homeomorphisms]{Classification of strict limits of planar BV homeomorphisms}

\author[D. Campbell]{Daniel Campbell}
\address{D.~Campbell: Department of Mathematics, University of Hradec Kr\' alov\' e, Rokitansk\'eho 62, 500 03 Hradec Kr\'alov\'e, Czech Republic} 
\address{Faculty of Economics, University of South Bohemia, Studentsk\' a 13, Cesk\' e Budejovice, Czech Republic}
\email{daniel.campbell@uhk.cz}

\email{}\author[A. Kauranen]{Aapo Kauranen}
\address{A.~Kauranen: Department of Mathematics and Statistics, University of Jyv\"askyl\"a, PL 35, 40014 Jyv\"askyl\"an yliopisto, Finland}
\email{aapo.p.kauranen@jyu.fi}

\author[E. Radici]{Emanuela Radici}
\address{E.~Radici: Institute of Mathematics, EPFL, CH-1015 Lausanne, Switzerland}
\email{emanuela.radici@epfl.ch}

\thanks{The first author was supported by the grant GACR 20-19018Y. The second author was supported by the Academy of Finland (project number 322441).}


\begin{abstract}
	We present a classification of strict limits of planar BV homeomorphisms. The authors and S. Hencl showed in a previous work \cite{CHKR} that such  mappings allow for cavitations and fractures singularities but fulfill a suitable generalization of the INV condition. As pointed out by J. Ball \cite{B}, these features are physically expected by limit configurations of elastic deformations. In the present work we develop a suitable generalization of the \emph{no-crossing} condition introduced by De Philippis and Pratelli in \cite{PP} to describe weak limits of planar Sobolev homeomorphisms that we call \emph{BV no-crossing} condition, and we show that a planar mapping satisfies this property if and only if it can be approximated strictly by homeomorphisms of bounded variations. 
\end{abstract} 

\maketitle

\section{Introduction}

In the recent years the problem of classifying the class of weak or strong limits of Sobolev diffeomorphisms gained a lot of attention due to its relvance in nonlinear elasticity and geometric function theory. 
Thanks to the pioneering work of Iwaniec and Onninen \cite{IO} and the more recent result of De Philippis and Pratelli \cite{PP}, the Sobolev classification in the planar setting is now well understood. 

More precisely, through the Sobolev diffeomorphic approximation result obtained in \cite{IKO1,IKO2} for $p > 1$, the authors of \cite{IO} show that the weak closure of $W^{1,p}$ homeomorphisms for $p\geq 2$ coincides with the respective strong closure of diffeomorphisms and the limit set is characterized by monotone Sobolev mappings. 

Relying on a different technique introduced in \cite{HP} for the diffeomorphic approximation of $W^{1,1}$ homeomorphisms, the authors of \cite{PP} can prove that the weak closure of $W^{1,p}$ homeomorphisms still coincides with the strong closure of diffeomorphisms for all $1 \leq p <\infty$. 
In this case the limit mappings may present discontinuities, moreover, monotonicity turns out to be too weak to describe the weak limits unless some more restrictive condition on the Jacobian is assumed. Thus, the authors need to introduce the new concept of \emph{no-crossing} condition. 
Intuitively speaking, a Sobolev mapping can be obtained as a limit of homeomorphisms if and only if the restriction of the map to almost any grid inside the domain can always be injectified while remaining uniformly close to the original map. This condition is flexible enough to be stated for mappings allowing for a $\mathcal{H}^1$-negligible set of discontinuities, as it is proved to be the case for $W^{1,p}$ limits when $p$ is smaller than the dimension of the domain. Indeed, as explicitly shown in \cite{PP}, weak limits of planar Sobolev homeomorphisms can present cavitations. 

A natural question would then be to consider the closure of planar homeomorphisms in the $BV$ setting so to include more complicated  discontinuities in the limit class. 

On the other hand, in his pioneering works \cite{B1,B2} Ball studied continuity and invertibility properties of mappings which can serve as energy minimizing deformations in elasticity theory. Being an elastic deformation a reversible shape change of the material, it is appropriate to describe elastic deformations as the class of homeomorphisms which map a reference configuartion onto a target configuration, eventually with prescribed boundary conditions. When the minimization of standard energy functionals does not admit solutions within the class of homeomorphisms, one is led to consider suitable relaxations of the problem in classes that still model the expected behaviour of elastic defomrations, namely the non interpenetration of the material. 

In this spirit, the INV condition introduced by M\" uller and Spector in \cite{MS} describes mappings $f: \er^n \to \er^n$ for which, loosely speaking, the image of $f(B(x,r))$ lies inside $f(\partial B(x,r))$ and the image of $f(\er^n \setminus B(x,r))$ remains outside $f(\partial B(x,r))$. Let us remark that singularities as cavitations (see \cite{GL}, Figure 4 for physical observation of cavitations) fulfill the INV condition, moreover, in the planar case, Sobolev weak limits of diffeomorphisms are always INV mappings. Since in many relevant situations a deformed material may break (see \cite{GL}, Figure 4 for physical observation of fractures), Ball proposed to generalize the mathematical model so to allow both for cavitations and fracture singularities.

A possible approach is to introduce energy functionals with an extra term accounting for the energy of the surface created by the deformations. Henao and Mora Corral study energies of this form in a sequence of works \cite{HMC1, HMC2, HMC3, HMC4} and show that the minimizers are one-to-one almost everywhere and can exhibit fractures.

An alternative approach was proposed by the authors and Hencl in \cite{CHKR}, motivated by the intereseting $BV$ energy relaxation results obtained by Kristensen and Rindler \cite{KR} and Rindler and Shaw \cite{RS} for fixed boundary conditions, see also \cite{BKK} for the Neumann case, and by the diffeomorphic approximation results for planar homeomorphims of bounded variations studied by the third author and Pratelli in \cite{PR1} and \cite{PR2}. The weak $BV$ topologies considered in the above mentioned works are the \emph{strict} and \emph{area-strict} ones. 
A sequence $f_k: \Omega \to \er^n$ of $BV$ functions is \emph{strictly} converging to $f \in BV(\Omega, \er^n)$ if $f_k \to f$ in $L^1(\Omega, \er^n)$ and $|Df_k|(\Omega) \to |Df|(\Omega)$; and is \emph{area-strictly} converging if, in addition, it is possible to decompose $Df_k$ as the sum of two measures $\mu_k + \nu_k$ such that $|\mu_k - D^a f|(\Omega) \to 0$ and $|\nu_k|(\Omega) \to |D^s f|(\Omega)$, where $D^a f$ and $D^s f$ denote the absolutely and the singular part of $Df$ respectively. Clearly, area-strict convergence implies the strict one. 

 In \cite{CHKR} we show that strict limits of planar $BV$ homeomorphisms can have fractures but still preserve a sort of monotonicity property. The INV condition cannot be formulated for such complicated singularities, however, we introduce a careful generalization of the concept of topological image via multifunctions and we are able to show that strict limits map disjoint sets onto essentially disjoint sets. Thus the strict (and the stronger area-strict) topology seems to be appropriate to describe $BV$ relaxations which are physically relevant in elasticity. 
 
The present manuscript complements the results of \cite{CHKR}. Indeed, in \cite{CHKR} the authors and Hencl study the basic properties of strict limits of planar $BV$ homeomorphisms and classify the admissible singularities but no full characterization of the limit class is provided. In the present work we introduce a suitable $BV$ generalization of the concept of \emph{no-crossing} condition introduced in \cite{PP} that we call \emph{no-crossing-BV} (NCBV) condition (see the precise Defintion \ref{DefNCBV}) and we show that it is flexible enough to characterize the class of strict limits of planar $BV$ homeomorphisms. 
Roughly speaking, a map of bounded variation satisfies the NCBV condition if any reasonable parametrization of the topological image of a grid can be injectified remaining remaining close in the uniform sense.
 
Our main result is the following:

\begin{thm}\label{main}
Let $f \in BV(Q(0,1), \er^2)$ be a planar $BV$ map that satisfies the NCBV condition and coincides with the identity on $\partial Q(0,1)$ then there exists a sequence $(f_k)_k \subset BV(Q(0,1), \er^2)$ of homeomorphisms extending the identity on $\partial Q(0,1)$ such that 
\begin{equation}\label{Manhattan strict convergence}
(|D_1f_k| + |D_2f_k|)(Q(0,1)) \longrightarrow ( |D_1f|+ |D_2f|)(Q(0,1)) \quad \mbox{ as $k \to \infty$.}
\end{equation}
Also, if $(f_k)_k \subset BV(Q(0,1), \er^2)$ is a sequence of planar BV homeomorphisms with $f_k(x) = x$ on $\partial Q(0,1)$ and satisfying \eqref{Manhattan strict convergence} for some $f \in BV(Q(0,1), \er^2)$, then the map $f$ satisfies the NCBV condition.
	\end{thm}

The second implication is proved by carefully chopping up the grid and choosing an appropriate parametrization of the image of $f_k$ on each of these parts. The first implication of Theorem \ref{main} is more involved, thus we conclude this introduction by briefly outlining the basic plan of its proof and the intuition behind the main steps.  

We will show that any NCBV map can be approximated in the sense of \eqref{Manhattan strict convergence} by finitely piecewise affine homeomoprhisms extending the identity on the boundary of $Q(0,1)$. The construction of the homeomorphism occurs in two steps: we first apply the NCBV condition and find a continuous and one-to-one mapping on a one dimensional grid of rectangles inside $Q(0,1)$. Without loss of generality we can assume this mapping to be piecewise linear on the grid, this is anyway true up to a small error in $L^\infty$.  

Once a continuous and piecewise linear map on the grid is defined, we can perform a piecewise affine homeomorphic extension separately in each rectangle 
being careful to comply with convergence \eqref{Manhattan strict convergence}. The idea is to use the main result of \cite{PR1} which provides a homeomorphic extension $v$ inside the rectangle $R$ having the minimal possible variation $(|D_1 v| + |D_2 v|)(R)$  once the boundary values are a fixed Jordan curve in $\er^2$. 
The minimal extension is based on the concept of geodesic filling, therefore, if the behaviour of the boundary values is close enough to that of the original NCBV map $f$, then we can find an estimate like 
\begin{equation}\label{prova}
(|D_1 v| + |D_2 v|)(R) \leq (|D_1 f| + |D_2 f|)(R) + \eps |R|
\end{equation}
and we can conclude \eqref{Manhattan strict convergence} thanks to the lower semicontinuity of the total variation. 
In order to achieve the desired control on the boundary values, we will have to apply the NCBV condition on a grid which is much finer than the one we are really interesetd in. This finer grid will include many extra lines which we will call \emph{guidelines} in the sequel. Intuitively speaking, the reason is the following.

 \begin{figure}[thbp]
\scalebox{0.6}{
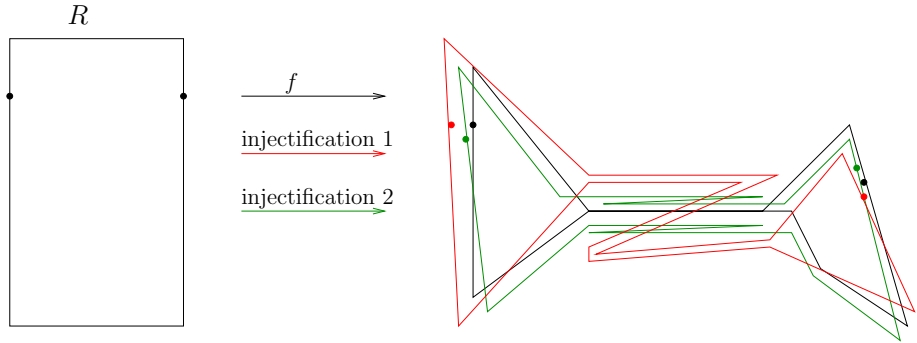
}
\caption{Possible different injectifications of a degenerate Jordan curve}\label{Fig:intro}
\end{figure}

If the topological image of $R$ through the original map $f$ is a degenerate Jordan curve (see Figure \ref{Fig:intro}), hence it collapses on itself in some parts, then corresponding geodesics inside $f(R)$ and inside the polygon identified by the injective curve provided by the NCBV condition may behave very differently. For example, both the green and the red polygons depicted in Figure \ref{Fig:intro} are possible injectifications of the map $f$, but in the red case the geodesic connecting the red spots inside the red polygon is much longer than the geodesic connecting the corresponding black spots inside $f(R)$ (the black degenerate polygon in Figure \ref{Fig:intro}). Hence, if the injectification gives the red curve, it may not be possible to achieve estimate \eqref{prova}. Even if the estimate \eqref{prova} does hold it is far from obvious how to prove it.

We solve this problem using the guidelines described above. The guidelines form a much finer grid of lines inside the rectangle $R$ and are chosen to be almost optimal in terms of the length of their images in $f$ on the interval they are chosen on. Then for the vast majority of pair of (horrizontally or vertically) opposing points we estimate the length of the geodesic connecting the image of the pair inside the injectification of $f$ on $R$ by the length of the guideline plus some tiny $\epsilon$. The set of pairs where this estimate does not hold is made so small that its final contribution to the variation is at most $\epsilon$. Thus we are able to obtain \eqref{prova}.

The paper is organized as follows. In section 2 we collect some useful definitions and preliminary results not new in the literature. In section 3 we introduce the \emph{no-crossing-BV} condition and we show that it well behaves with respect to monotonicity. Section 4 is entirely devoted to the proof of Theorem \ref{main}.

\section{Preliminaries}

In this section we collect some existing preliminary results that will be useful in the sequel.

\begin{lemma}\label{Inevitable}	
	Let $A,B\in \R^2$ and $\delta>0$ and $L:=\abs{A-B}$ Let $C\in B(A,\delta L)$ and $D\in B(B,\delta L).$
	Let $\eta:[0,1]\to \R^2$ be a path (with constant speed parametrization) joining points $C$ and $D$ with arc length
	$l(\eta)\leq (1+\epsilon)L.$ Let $\gamma:[0,1]\to \R^2$ be the constant speed parametrization of the line segment joining $A$ and $B.$
	
	Then for every $t\in [0,1]$ $\abs{\eta(t)-\gamma(t)}\leq 3\sqrt{2\epsilon+\delta} L.$

\end{lemma}
\begin{proof}
	To prove the claim we assume without loss of generality that $A$ is the origin and $B=(L,0).$ 
	Fix a point $t\in [0,1].$ We use notation $l_1=l(\eta|_{[0,t]})$ and $l_2=l(\eta|_{[t,1]}).$ Notice that $\eta|_{[0,t]}\subset B(A,l_1+\delta L)\subset B(A,t(1+\epsilon)L+\delta L)$ and therefore
	we have that the angle $\sphericalangle(\eta(t),(t(1+\epsilon)L+\delta L,0),(L,0))$ is in $[\pi/2,\pi].$ 
	This and  law of cosines give 
	\begin{equation}
	l_2^2\geq L^2((1-t)-\epsilon-\delta)^2+E^2,
	\end{equation}
	where $E$ is the distance between $\eta(t)$ and the point $(t(1+\epsilon)L+\delta L,0).$ Combining with the assumption on $l(\eta)=l_1+l_2\leq L(1+\epsilon)$ we have
	$$
	(1-t)L+\epsilon L\geq l_2\geq \sqrt{((1-t)L-t\epsilon L-\delta L)^2+E^2}.
	$$
	With some algebra this simplifies to 
	$$
	2((1-t)L-t\epsilon L-\delta L)((1+t)\epsilon L+\delta L)+((1+t)\epsilon L+\delta L)^2\geq E^2
	$$
	and further we obtain
	\begin{equation}
	3L^2(2\epsilon+\delta)\geq E^2.
	\end{equation}
	To finish we notice that $\abs{\eta(t)-\gamma(t)}\leq E+\delta L.$
\end{proof}

\subsection{Properties of weak limits}

In this part we recall useful properties of weak $BV$ strict limits of planar homeomorphisms. 

\begin{lemma}\label{Owens}
	Let $f_k,f\in BV(a,b)$ be such that $f_k$ converge strictly to $f$ on $(a,b)$. Further let $c\in(a,b)$ such that $f$ is continuous at $c$. Then $f_k$ converge strictly to $f$ on $(a,c)$ and on $(c,b)$.
\end{lemma}
\begin{proof}
	Since $f$ is continuous at $c$ we have $|Df|(\{c\}) = 0$. It is clear that $f_k\rightarrow f$ in $L^1$ on both $(a,c)$ and $(c,b)$. Therefore it remains to prove the 
	convergence of $|Df|(a,c)$ and $|Df|(c,b)$.

	Both of the sequences $|Df_k|(a,c)$ and $|Df_k|(c,b)$ are bounded and therefore have converging 
	subsequences. Choose any subsequence $k(j)$ so that 
	$$
	|Df_{k(j)}|(a,c)\rightarrow A
	\text{ and }
	|Df_{k(j)}|(c,b)\rightarrow B
	$$
	for some $A$ and $B$. Then we have 
	$$
	A + B \leq\lim_{j\rightarrow \infty} |Df_{k(j)}|(a,b)=|Df|(a,b).
	$$ 
	
	By the lower semi-continuity of the total variation we see that
	\begin{equation}\label{Contra}
	A\geq|Df|(a,c)\text{ and }B\geq|Df|(c,b).
	\end{equation}
	Now, assuming that $A>|Df|(a,c)$ and utilizing $|Df|(\{c\}) = 0$ we have 
	$$
	|Df|(c,b)=|Df|(a,b)-|Df|(a,c)\geq A+B -|Df|(a,c)>B.
	$$
	But this cannot be because of \eqref{Contra} and so
	$$
	|Df_k|(a,c)\rightarrow |Df|(A) 
	\text{ and }|Df_k|(c,b)\rightarrow |Df|(B).
	$$  
\end{proof}

We recall a known property of planar $BV$ functions. If $f: Q(0,1) \to \R^2$ is a $BV$ mapping, then for a.e. $t \in [-1,1]$ the restriction of $f$ to the line $\{y=t\}$ is one-dimensional $BV$. In particular, $f|_{\{y= t\}}$ has at most countably many jumps whose total size is finite. 
For the reader's convenience, we report here a reformulation of the statement Proposition 2.3 of \cite{CHKR} adapted to our setting. 

\begin{prop}\label{Proposition 2.3}
	Let $f_k, f \in BV(Q(0,1),\er^2)$ be a sequence of mappings such that 
	$$
	\lim_{k \to \infty} (|D_1 f_k| + |D_2 f_k|)(Q(0,1)) = (|D_1 f| + |D_2 f|)(Q(0,1)).
	$$
	Then there exists a $\mathcal{L}^1$-negligible set $N \subset [-1,1]$ such that, up to a subsequence, $f_k |_{\{y = t\}},\, f_k |_{\{ x = t\}},\, f |_{\{y = t\}},\, f |_{\{ x = t\}}$ are one dimensional $BV$ functions and 
	\begin{align*}
	& |D_1 f_k |_{\{y = t\}}| ([-1,1] \times\{t\}) \longrightarrow |D_1 f |_{\{y = t\}}| ([-1,1] \times\{t\}), \\
	& |D_2 f_k |_{\{x = t\}}| (\{t\} \times [-1,1]) \longrightarrow |D_2 f |_{\{x = t\}}| (\{t\} \times [-1,1])
	\end{align*}
	for all $t \in [-1,1] \setminus N$.
\end{prop}

\subsection{Choosing grids}
	\begin{lemma}\label{GridChooser}
Let $f \in BV(Q(0,1))$. Then for $\mathcal{L}^2$-almost every choice of $(x,y)\in Q(0,1)$ it holds that $f_{|T}$ is continuous at $(x,y)$, where $T = \{x\}\times[-1,1]\cup[-1,1]\times\{y\}$.
	\end{lemma}
	\begin{proof}
		By $A$ denote the $x$ ordinates in $[-1,1]$ such that $f_{|\{x\}\times[-1,1]}$ is BV on the line and by $B$ denote the $y$ ordinates such that $f_{|[-1,1]\times\{y\}}$ is BV on those lines (see Proposition~\ref{Proposition 2.3}). Of course $\mathcal{L}^1([-1,1]\setminus (A\cap B)) = 0$. Let us denote a set of `bad' $x$ ordinates as $\tilde{A}$ defined as those
		$$
			\tilde{A} = \{x\in A: \mathcal{L}^1(\{y\in B: |D_1f|(x,y)>0\})>0 \}.
		$$
		
		Let us prove that $\mathcal{L}^2(\tilde{A}\times B) = 0$. For all $y\in B$ the set $\{x\in A: |D_1f|(x,y)>0\}$ is countable and so $\mathcal{L}^1(\{x\in A: |D_1f|(x,y)>0\})=0$. By Fubini theorem therefore $$
			\mathcal{L}^2(\tilde{A}\times B) = \int_{-1}^1\int_{-1}^1 \chi_{\{x\in A: |D_1f|(x,y)>0\}}(x,y) \, dy \,  dx = \int_{-1}^1 0 \,  dx = 0.
		$$
		This however implies that  $\mathcal{L}^1(\tilde{A}) = 0$. Similarly the set
		$$
			\tilde{B} = \{y\in B: \mathcal{L}^1(\{x\in A: |D_2f|(x,y)>0\})>0 \}
		$$
		has measure zero.

		Obviously $f$ is continuous at $(x,y)$ with respect to $T$ whenever $|D_1f|(x,y)+|D_2f|(x,y) = 0$. Further for all $x\in A$ it holds that $|D_2f|(x,y) = 0$ for almost all $y\in[-1,1]$. Then (by the definition of $\tilde{A}$) for any choice of $x \in A\setminus \tilde{A}$ we have that $f$ is continuous at $(x,y)$ with respect to $T$ for almost every choice of $y\in B$ or more specifically for any choice of $y\in B\setminus \tilde{B}$. Since $(A \setminus \tilde{A})\times (B \setminus \tilde{B})$ has full measure, we have our claim.
	\end{proof}

	\begin{corollary}\label{GridChooser2}
		Let $K \in \en$ and $f\in BV(Q(0,1); Q(0,1))$ then for $\mathcal{L}^{2K}$-almost every choice of $-1<x_1 < x_2< \dots, <x_K<1$ and $-1<y_1 < y_2< \dots, <y_K<1$ it holds that $f_{|T}$ is continuous at $(x_j,y_m)$ where $T = (\bigcup_{j=1}^K\{x_j\}\times[-1,1])\cup(\bigcup_{m=1}^K[-1,1]\times\{y_m\})$.
	\end{corollary}
	
	We conclude this part defining the concept of generalized segment, already introduced in \cite{PP}, which will be useful throughout the proof of Theorem \ref{main}.
	
	\begin{definition}[generalized segments]\label{generalized segments}
	Let $\mathcal{G} \subset \R^2$ be a connected one dimensional grid given by the boundaries of finitely many, non-degenerate rectangles. Let $R$ be one of such rectangles and $a,b$ be two different points of $\partial R$, hence also points of the grid $\mathcal{G}$. Given $\xi >0$ a small parameter, the \emph{generalized segment} $[ab]$ between $a$ and $b$ in $R$ is defined as the standard segment $ab$ if the two points are not in the same side of $\partial R$; otherwise, $[ab]$ is the union of two segments of the form $am$ and $mb$ where $m$ is the point inside $R$ whose distance from the side containing $a$ and $b$ is less than $\xi|ab|$ and the projection of $m$ on the segment $ab$ is exactly the mid-point.
	 \end{definition}

\subsection{Minimal $BV$ extension}

We use the following result of \cite{PR1}. Let $R=[a_1,a_2]\times [b_1,b_2]\subset \R^2$ and let 
$\varphi:\partial R\to \R^2$ be a continuous injection. Let
$\mathcal P$ be the domain bounded by the curve $\varphi(\partial R).$
By 
\begin{equation}
d_{\mathcal P}(p_1,p_2)=\inf\{l(\gamma)\colon \gamma \textrm{ is  
a path joining } p_1,\, p_2 \textrm{ in }\mathcal P\}.
\end{equation}
Here a path joining $p_1$ and $p_2$ in $\mathcal P$ is a continuous curve $\gamma:[0,1]\to \R^2$ such that $\gamma((0,1))\subset
\mathcal P$ and $\gamma(0)=p_1$ and $\gamma(1)=p_2.$ Notice that
$d_{\mathcal P}$ is well defined for points $p_i\in \overline{\mathcal P}.$
We define 
\begin{equation}\label{funzione Psi}
 \Psi(\varphi)=\int_{[a_1,a_2]}d_{\mathcal{P}}(\varphi(t,b_1),\varphi(t,b_2))\,dt+
 \int_{[b_1,b_2]}d_{\mathcal{P}}(\varphi(a_1,t),\varphi(a_2,t))\,dt
\end{equation}

\begin{thm}
 \label{MinExt}[Minimal Extension, \cite{PR1} Theorem A]
 Let $\varphi$ be as above. For every $\epsilon>0$ there exists a piecewise affine 
 homeomorphism $h$ defined on $R,$ $h|_{\partial R}=\varphi$ and
 \begin{equation}\label{stima MinExt}
  |D_1 h|(R)+|D_2 h|(R)\leq \Psi(\varphi)+\epsilon.
 \end{equation}
\end{thm}

\section{The NCBV property for $BV$ maps}

	In \cite{PP} the authors introduced a property called the $NC$ condition that characterises the limits of $W^{1,p}$ homeomorphisms from $Q(0,1)$ onto $Q(0,1)$ equalling the identity on the boundary. We aim to extend this concept such that it enables us to characterize the strict limits of $BV$ homeomorphisms. We call the generalization of their concept the $NCBV$ condition and define it using the $BV$ on lines characterisation. In this section we expound this concept.

	Let $f \in BV(Q(0,1), Q(0,1))$,  then there exists a pair of sets $G_1, G_2 \subset [-1,1]$ such that $\mathcal{L}^1\big([-1,1] \setminus G_1\big) = \mathcal{L}^1\big([-1,1] \setminus G_2\big) = 0$ and for all $x \in G_1$, $g_x(\cdot) = f(x,\cdot)$ is $BV$ on $[-1,1]$, as is $h_y(\cdot) = f(\cdot,y)$ for all $y \in G_2$. We define a pair of multi-function representatives $\tilde{f}_1,\tilde{f}_2 \subset [-1,1]^2\times [-1,1]^2$ of the $BV$ map $f$ as follows. Suppose that $x \in G_1$, then $f$ is $BV$ on $\{(x,t) ; t\in [-1,1]\}$. Similarly suppose that $y\in G_2$, then $f$ is $BV$ on $\{(t, y) ; t\in [-1,1]\}$. Let $A_{1,x,y} = \lim_{t\nearrow y}f(x,t)$ and $B_{1,x,y} = \lim_{t\searrow y}f(x,t)$ then $(x,y,w,z) \in\tilde{f}_1$ exactly when $(w,z) \in [A_{1,x,y} B_{1,x,y}]$ where $[AB]$ is the segment that connects $A$ to $B$. Similarly call $A_{2,x,y} = \lim_{t\nearrow x}f(t,y)$ and $B_{2,x,y} = \lim_{t\searrow x}f(t,y)$ then $(x,y,w,z) \in\tilde{f}_2$ exactly when $(w,z) \in [A_{2,x,y} B_{2,x,y}]$.
	
	For each $x\in G_1$ we can define a countable bad set $N_{1,x}$ outside of which we have $\tilde{f}_1(x,y) = f(x,y)$. For every $x\in G_1$ we have the decomposition of $g_x$ into its continuous and jump parts on $[-1,1]$, similarly for $h_y$ for all $y \in G_2$. Therefore for each $x$ such that $f$ is $BV$ on $\{(x,t) ; t\in [-1,1]\}$, we have a countable set $N_{1,x}$ such that for all $y\in[-1,1]\setminus N_{1,x}$ we have $\tilde{f}_1\cap\{(x,y,\er, \er)\} = \{(x,y,f(x,y))\}$. Similarly we define $N_{2,y}$ so that for all $x\in[-1,1]\setminus N_{2,y}$ we have $\tilde{f}_2\cap\{(x,y,\er, \er)\} = \{(x,y,f(x,y))\}$.

	For each $x\in G_1$ call 
	$$
		l^{f}_{1,x}(y) =  \frac{y+1 + |Dg_x|([-1,y))}{2+ |Dg_x|([-1,1])}.
	$$
	Thus $l^{f}_{1,x}$ is a strictly increasing function with values in $[0,1]$ and with jumps at each point $y\in N_{1,x}$. 
	Therefore for each $t \in l^{f}_{1,x}([-1,1]\setminus N_{1,x})$ we have a unique $y$ such that $l^{f}_{1,x}(y) = t$, call $(l^{f}_{1,x})^{-1}(t) = y$ for all $t \in l^{f}_{1,x}([-1,1]\setminus N_{1,x})$.
	Then we can define $\phi_{1,x}(t) = f(x,l^{-1}_{1,x}(t))$ for all $t\in l^{f}_{1,x}([-1,1]\setminus N_{1,x})$. We take a $y \in N_{1,x}$ and get a pair $a_{1,x,y} = \lim_{s \nearrow y}l^{f}_{1,x}(s)$ and $b_{1,x,y} = \lim_{s \searrow y}l^{f}_{1,x}(s)$. From the definition of $\phi_{1,x}$ we have that $\lim_{t\nearrow a_{1,x,y}}\phi_{1,x}(t) = A_{1,x,y}$ and $\lim_{t\searrow b_{1,x,y}}\phi_{1,x}(t) = B_{1,x,y}$. Then for $t\in [a_{1,x,y}, b_{1,x,y}]$ we define $\phi_{1,x}(t)$ as follows
	$$
		\phi_{1,x}(t) = \frac{A_{1,x,y}(b_{1,x,y}-t) + B_{1,x,y}(t-a_{1,x,y})}{b_{1,x,y} - a_{1,x,y}}.
	$$
	Thus defined $\phi_{1,x}$ is continuous. In fact $t \to (x,t, \phi_{1,x})$ is a bi-Lipschitz parametrization of the graph of $\tilde{f}_1$ restricted to $\{x\times[-1,1]\}$, i.e. $\{(x,t,\phi_{1,x}(t)); t\in[0,1]\} = \tilde{f}_1\cap (\{x\}\times \er^3)$. Similarly we construct $\phi_{2,y}$.

	\begin{definition}\label{good starting grid}[Good starting grid and geometrical representatives]
	Let the mapping $f\in BV(Q(0,1); Q(0,1))$, let $K \in \en$ and let $\{x_1,x_2 \dots, x_K\}\subset G_1$ with $\tfrac{1}{K}<x_{i+1} - x_i<\tfrac{4}{K}$ and $\{y_1, y_2 ,\dots, y_K\}\subset G_2$ with $\tfrac{1}{K}<y_{i+1} - y_i<\tfrac{4}{K}$ be a finite choice of horizontal and vertical coordinates be such that $g_{x_i}$ is continuous at each $y_j$, and $h_{y_j}$ is continuous at $x_i$ for all $i,j = 1 \ldots K$ (for existence see Corollary~\eqref{GridChooser2}). 
Then we call 
		$$
	\Gamma= \bigcup_{i=1}^K  \{x_i\} \times [-1,1]  \cup \bigcup_{j=1}^K  [-1,1] \times \{ y_j\}
		$$
		a \emph{good starting grid} for $f$. 
		We call $\gamma : \Gamma \to Q(0,1)$ the mapping defined on each segment $\{ (x_i,t) \in \Gamma,\, t \in [y_{j-1},y_j]  \}$ by
	\begin{equation}\label{GoingUp}
		\gamma(x_i,t) = \phi_{1,x_i}\Big( \tfrac{y_j -t}{y_j - y_{j-1}} (l^{f}_{1,x_i})^{-1}(y_{j-1}) + \tfrac{t-y_{j-1}}{y_j - y_{j-1}} (l^{f}_{1,x_i})^{-1}(y_{j})\Big).
	\end{equation}
	and on the segment $\{(t,y_j)\in \Gamma; t\in[x_{i-1}, x_i]\}$ by
	\begin{equation}\label{GoingRight}
		\gamma(t,y_j) = \phi_{2,y_j}\Big( \tfrac{x_i -t}{x_i - x_{i-1}} (l^{f}_{2,y_j})^{-1}(x_{i-1}) + \tfrac{t-x_{i-1}}{x_i - x_{i-1}} (l^{f}_{2,y_j})^{-1}(x_{i})\Big).
	\end{equation}
	the \emph{geometrical representative} of $f$ on $\Gamma$.
	\end{definition}
	
Given a good starting grid $\Gamma$ for $f$, we call the set
$$
	\bigcup_{i=1}^N \bigcup_{y \in \pi_2(V_i)} \big(x_i, y ,\tilde{f}_1(x_i,y)\big)\cup \bigcup_{j=1}^{M}\bigcup_{x \in \pi_1(H_j) } \big(x, y_j ,\tilde{f}_2(x,y_j)\big)
$$
the multifunction graph of $f$ on $\Gamma$, where $\pi_i(\cdot)$ is the usual projection on the coordinate axis. It is easy to observe that there exists a unique continuous bijection from the graph of the geometrical representative $\gamma$ of $f$ onto the multifunction graph of $f$ on $\Gamma$.
	
	\begin{definition}\label{DefNCBV}[No-Crossing $BV$ condition]
		We say that $f \in BV(Q(0,1), \er^2)$ with $f(x) = x$ on $\partial Q(0,1)$, satisfies the $NCBV$ condition if for every $\sigma >0$ and for every pair of good starting grid $\Gamma$ and respective geometrical representative $\gamma : \Gamma \to Q(0,1)$ of $f$ on $\Gamma$, there exists a continuous and one-to-one map $\gamma_{\sigma} : \Gamma \to Q(0,1)$ such that $\|\gamma - \gamma_{\sigma}\|_{L^\infty(\Gamma)}<\sigma$.
	\end{definition}

	\begin{remark}
		The definition of the NCBV condition is independent on a bilipschitz reparameterization of the grid $\Gamma$. Sepcifically let $\psi$ be a piecewise linear continuous and injective function mapping $\Gamma$ onto $\Gamma$. Then we can define $\phi_{\psi}$, a reparametrisation of $\gamma(\Gamma)$ as
		$$
			\phi_{\psi}(t) = \gamma(\psi(t))
		$$
		and similarly we define a reparametrisation of $\gamma_{\sigma}(\Gamma)$, which we will call $\phi_{\psi, \sigma}$ as
		$$
			\phi_{\psi, \sigma}(t) = \gamma_{\sigma}(\psi(t)).
		$$
		Obviously, since $\gamma_{\sigma}$ and $\psi$ are both injective we have that $\phi_{\psi,\sigma}$ is injective and
		$$
		\|\phi_{\psi} - \phi_{\psi, \sigma}\|_{\infty}<\sigma.
		$$
		That is to say, given a $f$ satisfying NCBV and any grid $\Gamma$ we can find an injectification of any reparametrization of the image of $\Gamma$.
	\end{remark}

	\begin{prop}\label{monotonicity}
		Let $f\in BV(Q(0,1),Q(0,1))$ satisfy the NCBV condition.
		Let $-1<x_1<x_2<1$ and $-1<y_1<y_2<1$ be such that $f_{|\{x_i\}\times[-1,1]}$ and $f_{|[-1,1]\times\{y_i\}}$ are BV functions. Call $R = \{(x,y); x_1\leq x\leq x_2, y_1\leq y \leq y_2\}$, then let $\Gamma$ be a good starting grid such that $\partial R \subseteq \Gamma$ and let $\gamma$ be the geometrical representative of $f$ on the good starting grid $\Gamma$. Then, for $\mathcal{L}^2$-almost every $(x,y)$ in the interior of $R$, it holds
		$$
			f(x,y) \in \cR
		$$
		where $\cR=\gamma(\partial R) \cup \{z\in Q(0,1); \deg(z, \gamma, R) \neq 0\}$.
	\end{prop}
	\begin{proof}
	
		Let us choose any $(x,y)$ such that $f$ is BV on $\{x\}\times[-1,1]$ and $[-1,1]\times \{y\}$ and $f_{|\{x\}\times[-1,1]\cup[-1,1]\times \{y\}\cup \Gamma}$ is continuous at all the crosses of $\{x\}\times[-1,1]\cup[-1,1]\times \{y\}\cup \Gamma$. By Corollary~\ref{GridChooser2} this is almost all $(x,y)$ on the interior of $R$. We denote $\tilde{\Gamma}:= \Gamma \cup \{x\}\times[-1,1] \cup [-1,1]\times \{y\}$, which is again a good starting grid, and by $\tilde{\gamma}$ we denote the geometric representative of $f$ on $\tilde{\Gamma}$. Then by the NCBV property we have, for every $\sigma>0$, an injective continuous uniform approximation of $\tilde{\gamma}$ on $\tilde{\Gamma}$ called $\tilde{\phi}_{\sigma}$.
		
		Of course $\tilde{\phi}_{\sigma}(\partial R)$ is a Jordan curve in $Q(0,1)$. Obviously for every $\sigma>0$ it holds that $\tilde{\phi}_{\sigma}(x,y)$ lies in the interior of $\tilde{\phi}_{\sigma}(\partial R)$ (call this set $U$). This is because any point in $\partial Q(0,1)$ lies outside the Jordan curve and the curve $\tilde{\gamma}(\{x\}\times[y,1])$ starts in $\tilde{\phi}_{\sigma}(x,y)$ and ends in $\partial Q(0,1)$.
		
		It holds that $U \subset \cR+B(0,\sigma)$ and $f(x,y)\in B(\tilde{\phi}_\sigma(x,y), \sigma)$ and therefore $f(x,y)\in\cR+B(0,2\sigma)$ for all $\sigma>0$. Since $\cR$ coincides with the closure of the the bounded domain identified by the Jordan curve $\gamma(\partial R)$, it holds that $f(x,y) \in \cR$. 
	\end{proof}

	\begin{definition}\label{good arrival grid}[Good arrival grids]
		Let the mapping $f\in BV(Q(0,1); Q(0,1))$ and let $\Gamma$ be a good starting grid for $f$ and let $\gamma$ be the geometrical representative of $f$ on $\Gamma$. Let $\eta>0$, and some coordinates $-1=x_0<x_1<x_2\, \cdots\, ,\,< x_N=1$ and $-1=y_0<y_1<y_2\, \cdots < y_M=1$ with $x_{n+1}-x_n<\eta$ and $y_{m+1}-y_m<\eta$ for every $0\leq n< N$ and $0\leq m<M$, we say that
		\begin{equation}\label{defgag}
		\G = \bigcup_{n=1}^N \{x_n\}\times [-1,1]\ \cup\ \bigcup_{m=0}^M [-1,1] \times \{y_m\}\subseteq Q(0,1)
		\end{equation}
		is a \emph{good arrival grid for $f$ associated with $\Gamma$ and with side-length $\eta$} if $\gamma^{-1}(\G)\cap \Gamma \cap\textit{ int}\, Q(0,1)$ is a set of finitely many points $p$, each of which satisfies the following
		\begin{itemize}
			\item[$\cdot$] $p$ is not a cross of the grid $\Gamma$,
			\item[$\cdot$] $\gamma(p)$ is not a cross of the grid $\G$,
			\item[$\cdot$] $p$ is a point where the (strong 1-dimensional) derivative $D_{\tau}\gamma$ (w.r.t $\Gamma$) exists,
			\item[$\cdot$] $D_\tau\gamma(p)$ is not parallel to the side of $\G$ containing $\gamma(p)$.
		\end{itemize}
	\end{definition}

	An important fact is that good arrival grids always exist. More precisely, we have the following property, whose proof is a simple variant of the proof of~\cite[Lemma~3.6]{PP} and can be found in~\cite[Lemma~4.4]{CPR}. 
	
	\begin{lemma}\label{ArrivalGrid}
Let $f\in BV(Q(0,1); Q(0,1))$ and let $\Gamma$ be a good starting grid for $f$ in $Q(0,1)$. Then the geometrical representative $\gamma$ of $f$ on $\Gamma$ is in $W^{1,1}(\Gamma, Q(0,1))$. Moreover, there exists $\bar{\eta} = \bar{\eta}(L) >0$ such that for any $0 < \eta < \bar \eta$ and any $\Sigma \subset \Gamma$ $\H^1$-negligible set, there exists a good arrival grid $\G$ for $f$ associated with $\Gamma$, with side-length $\eta$, and such that $\gamma^{-1}(\G)\cap \Sigma=\emptyset$.
	\end{lemma}

\section{Proof of Theorem~\ref{main}}

 \subsection{Strict convergence implies NCBV}

In this section we deal with the simpler implication of Theorem \ref{main}, namely we prove the following statement.

If $f_k,\, f \in BV(Q(0,1), \er^2)$ coinciding with the identity on $\partial Q(0,1)$ and $f_k$ are homeomorphisms such that 
\[ 
(|D_1f_k| + |D_2f_k|)(Q(0,1)) \longrightarrow ( |D_1f|+ |D_2f|)(Q(0,1)) \quad \mbox{ as $k \to \infty$,}
\]
then $f$ satisfies NCBV in the sense of Definition \ref{DefNCBV}. 

The proof uses elements from \cite[Proposition 4.1]{CHKR}. We modify those elements in the proof rather than just referring to them.

\begin{proof}
Thanks to Proposition \ref{Proposition 2.3}, the sequence $f_k$ converges strictly on almost every line parallel to coordinate axes. Call $G_1$ the set of $x$-coordinates such that $f_k$ converge on the lines $\{x\}\times[-1,1]$ for $x\in G_1$ and call $G_2$ the set of $y$-coordinates such that $f_k$ converge on the lines $[-1,1]\times\{y\}$ for $y\in G_2$.

Take any finite selection $x_1,x_2,\dots x_n \in G_1$ and $-1 = x_0<x_1<\dots < x_n <x_{n+1} = 1$ and $y_1,y_2, \dots y_m \in G_2$ and $-1=y_0< y_1 < \dots < y_m < y_{m+1} = 1$ in such a way (see Corollary~\ref{GridChooser2}) that $f$ is continuous at each $(x_i, y_j)$ with respect to the set $\Gamma$ which we define as
$$
\Gamma  = \bigcup_{i=1}^n \{x_i\}\times [-1,1] \cup \bigcup_{j=1}^m[-1,1]\times \{y_j\}.
$$

	Let $\sigma > 0$ be a fixed positive number. We separate the jumps of $f_{|\Gamma}$ into two categories small $\mathcal{S} = \{(x,y); |D_{\tau}f|(x,y) < \sigma/5 \}$ and big $\mathcal{B} = \{(x,y); |D_{\tau}f|(x,y) \geq \sigma/5 \}$, where $D_{\tau}$ is the derivative in the tangential direction to $\Gamma$ at $(x,y)$ (at $(x_i,y_j)$ vertices of $\Gamma$ both $|D_1 f|(x_i,y_j) = |D_2 f|(x_i,y_j) = 0$ because $f$ is continuous at the vertices). Notice that the set $\mathcal{B}$ is finite because $f_{|\Gamma}$ is BV on $\Gamma$.
	
	For each point in $(x,y_j) \in \mathcal{B}$ (the case for $(x_i,y)\in \mathcal{B}$ is similar) we define a segment containing $(x,y_j)$ as follows. Let $0\leq i \leq n$ be such that $x_i<x<x_{i+1}$. We choose two points $x^-, x^+$ such that $x_i <x^-<x<x^+<x_{i+1}$. Since there are a finite number of such $(x,y_j) \in \mathcal{B}$ we can choose $x^-,x^+$ such that the segments $[(x^-,y_j)(x^+,y_j)]$ are pairwise disjoint. Further we may assume that $f$ is continuous at $x^-$ and $x^+$ becasue $f$ is continuous at almost every point of $[-1,1] \times \{y_j\}$. Since
	$$
		\lim_{t\nearrow x}|D_1f|((t,x)\times\{y_j\}) = \lim_{t\searrow x}|D_1f|((x,t)\times\{y_j\}) = 0
	$$
	we may also assume that
	$$
		|D_1f|((x^-,x)\times\{y_j\}) <\frac{\sigma^3}{100} \quad \text{and} \quad|D_1f|((x,x^+)\times\{y_j\}) <\frac{\sigma^3}{100}.
	$$
	Because $|D_1f|(x^{\pm},y_j) = 0$ we have that $(f_k) |_{(x^-,x^+)\times\{y_j\}}$ converges strictly on $(x^-,x^+)\times\{y_j\}$ by Lemma~\ref{Owens}.
	
	After subtracting these segments which contain all the `big' jumps, we are left with the remaining part of $[-1,1]\times\{y_j\}$, which contains only small jumps which have size less than $\sigma/5$. Therefore we split the remaining part of $[-1,1]\times\{y_j\}$ into segments $(x^-,x^+)\times \{y_j\}$ such that $|D_1f|((x^-,x^+)\times\{y_j\})<\sigma/4$. Thus on each $[-1,1]\times\{y_j\}$ we have a finite number of open segments (and similarly on each $\{x_i\}\times[-1,1]$) which cover $[-1,1]\times\{y_j\}$ (resp. $\{x_i\}\times[-1,1]$) except for a finite number of endpoints at which $f$ is continuous. There exists a minimum number $d>0$ such that the length of any segment we chose on the grid $\Gamma$ is at least $d$.
	
Observe that Lemma~\ref{Owens} ensures that $f_k$ converges to $f$ strictly on each of the chosen segments because, by construction, $f$ is always continuous at the endpoints. In particular, this means we have $\|f_k - f\|_{L^1(I)}\to 0$ where $I$ is any one of the chosen segments.

Let us first concentrate on a chosen segment $I$ not containing any big jump, hence $I \cap \mathcal{B} = \emptyset$.
Since $|D_\tau f|(I) < \sigma/4$ and $f_k$ is strictly converging to $f$ on $I$, there exists some $k_0$ such that $|D_{\tau}f_k|(I) < \sigma/3$ and $\|f_k - f\|_{L^1(I)} < d \sigma/4$ for any $k \geq k_0$. Since the number of segements $I$ is finite we may assume that the threshold $k_0$ is such that the above holds for all the segments $I$ not containing big jumps simultaneously. Since $\mathcal{H}^1(I)\geq d$, we deduce from $\|f_k - f\|_{L^1(I)} < d \sigma/4$ that $\{|f_k - f| \leq \sigma/4\} \neq \emptyset$ for all $k\geq k_0$. Therefore 
	$$
	\| f_k - f \|_{L^{\infty}(I)} \leq \sigma/4 +  \sigma/4+  \sigma/3  < \sigma
	$$
	because there is at least a point $(a,b) \in I$ for which $|f_k(a,b) - f(a,b)| < \sigma/4$, moreover $|D_\tau f|(I) < \sigma/4$, $|D_\tau f_k|(I) < \sigma/3$ for $k\geq k_0$. 
	
	We denote $\gamma$ the constant speed parametrization of the geometric representative of $f$ on $I$ and $\gamma_k$ the constant speed parametrization of $f_k(I)$. Then for $I$ which do not contain big jumps we have
	\begin{equation}\label{SmallJumps}
		\|\gamma_k - \gamma\|_{L^{\infty}(I)} < 2 \sigma.
	\end{equation}
	
	Let us now focus on a chosen segment $I= (x^-,x^+)\times\{y_j\}$ which interescts $\mathcal{B}$, thus $I \cap \mathcal{B} = \{(x,y_j)\}$ is a big jump for $f$. 
	By the strict convergence on $I$ (and up to further increase $k_0$,) we have that
	\begin{equation}\label{WakeUp}
		\begin{aligned}
		|D_1f_k|((x^-,x^+)\times\{y_j\}) 
		&\leq |D_1f|((x^-,x^+)\times\{y_j\}) + \sigma^3/100\\
		& \leq |D_1f|(\{(x,y_j)\}) + \sigma^3/25
		\end{aligned}
	\end{equation}
	for $k\geq k_0$. Indeed, $|D_1f|((x^-,x)\times\{y_j\}) < \sigma^3/100$ and $|Df|((x,x^+)\times\{y_j\}) < \sigma^3/100$. Let now $w^-, w^+$ be such that $x^-<w^-<x<w^+<x^+$ and $f$ is continuous at $w^-$ and $w^+$. Then $f_k$ converges strictly to $f$ on $(x^-,w^-)$ and $(w^+,x^+)$ by Lemma~\ref{Owens} and, since $(x^-,w^-) \subset (x^-,x)$ (resp. $(w^+,x^+) \subset (x,x^+)$), if $k_0$ is big enough then $|D_1 f_k|((x^-,w^-)\times\{y_j\}) < \sigma^3/50$ (resp. $|D_1 f_k|((w^+,x^+)\times\{y_j\}) < \sigma^3/50$) for all $k\geq k_0$.
	 Moreover, the segments $(x^-,w^-)$ and $(w^+,x^+)$ do not intersect $\mathcal{B}$ and (up to increase $k_0$) the above argument ensures that there exist some points $a_k\in (x^-,w^-)$ and $b_k\in(w^+,x^+)$, with $f |_{\Gamma}$ continuous at $(a_k,y_j)$ and $(b_k,y_j)$, such that
	\begin{equation}\label{GrapeVine}
		|f_k(a_k,y_j) - f(a_k,y_j)|< \sigma^3 /100 \quad \text{ and } \quad |f_k(b_k,y_j) - f(b_k,y_j)|< \sigma^3 /100.
	\end{equation}
	Therefore 
	$$
		\left|f_k(a_k,y_j) - \lim_{t\nearrow x} f(t,y_j) \right| \leq |f_k(a_k,y_j) - f(a_k,y_j)| + |D_{1}f|((x^-,x)\times\{y_j\}) \leq \sigma^3 /50
	$$
	thus also
	\begin{equation}\label{BigBrassBed}
		\begin{aligned}
			|f_k(x^-,y_j) - \lim_{t\nearrow x}f(t,y_j)| \leq \sigma^3 /50+|D_1f_k((x^-,w^-)\times\{y_j\})| \leq \sigma^3 /25.
		\end{aligned}
	\end{equation}
	A similar argument holds for $|\lim_{t \searrow x} f(t,y_j) - f_k(x^+,y_j)|$.

	Let now $\gamma$ be the constant speed parametrization of the geometric representative of $f$ on $(x^-,x^+)\times \{y_j\}$ and let $\gamma_k$ be the constant speed parametrization of the curve $f_k((x^-,x^+)\times \{y_j\})$. Then both $\gamma$ and $\gamma_k$ are curves parametrized at constant speed from points $C =f(x^-)$ (resp $C_k = f_k(x^-)$) to $D = f(x^+)$ (resp $D_k = f_k(x^+)$). Thanks to \eqref{BigBrassBed}, we know that the points $C,C_k$ belong to $B(A,\sigma^3/25)$ where $A = \lim_{t\nearrow x}f(t,y_j)$ and $D,D_k \in B(B,\sigma^3/25)$ where $B = \lim_{t\searrow x}f(t,y_j)$. Moreover, from $|D_1f|((x^-,x)\times\{y_j\}) < \sigma^3/100$, $|Df|((x,x^+)\times\{y_j\}) < \sigma^3/100$ and $(x,y_j) \in \mathcal{B}$, we deduce that $l(\gamma |_{(x^-,x^+)\times \{y_j\}}) \leq |A-B| + \sigma^3 / 25$ and from \eqref{WakeUp} we also conclude that $l(\gamma_k |_{(x^-,x^+)\times \{y_j\}}) \leq  |A-B| + \sigma^3 / 25$ for all $k \geq k_0$. 
	
On the other hand, the fact that $(x,y_j) \in \mathcal{B}$ ensures that $|A-B| \geq \sigma / 5$ thus also $l(\gamma |_{(x^-,x^+)\times \{y_j\}}),\,  l1(\gamma_k |_{(x^-,x^+)\times \{y_j\}}) \leq |A-B|(1+ \eps^2/5)$. 

Therefore, by Lemma~\ref{Inevitable} we have
	\begin{equation}\label{BigJumps}
		\|\gamma - \gamma_k\|_{L^{\infty}((x^-,x^+)\times \{y_j\})} < 6\sqrt{2\sigma^2/5+\sigma^2/5} |A-B| \leq 24\sigma,
	\end{equation}
	where in the last estimate we used that $|A-B|<4$.
	
	Gathering together \eqref{SmallJumps} and \eqref{BigJumps} and reparametrizing $\gamma$ by constant speed parametrization on each segment $(x_i,x_{i+1})\times\{y_j\}$ and each $\{x_i\}\times (y_j, y_{j+1})$, we get a new $\gamma_k$ and a new $\gamma$ satisfying
	$$
		\|\gamma_k - \gamma\|_{L^\infty(\Gamma)} < 24\sigma.
	$$

	Since $f_k$ is a homeomorphism, $\gamma_k$ is injective and continuous and hence $f$ satisfies the NCBV condition.
\end{proof}

\subsection{Approximation of NCBV map}
Now let us take any $f\in BV(Q(0,1), \er^2)$ satisfying the NCBV condition and coinciding with the identity on $\partial Q(0,1)$, we prove that it can be approximated by $BV$ homeomorphisms in the sense of \eqref{Manhattan strict convergence}. 

\begin{proof}
\step{1}{Choice of a good starting grid with oscillation estimates}{start}

Let $\epsilon>0$ (without loss of generality we may assume that $\epsilon<\tfrac{1}{100}$) and find a $K\in \en$ such that
$$
	\int_{-1+\tfrac{2}{K+1}(m-1)}^{-1+\tfrac{2}{K+1}m} |Df_1|([-1,1]\times\{y\}) \, dy < \epsilon^2
$$
for all $m=1, 2, \dots, K+1$ and
$$
\int_{-1+\tfrac{2}{K+1}(j-1)}^{-1+\tfrac{2}{K+1}j} |Df_2|(\{x\}\times[-1,1]) \, dx < \epsilon^2
$$
for all $j=1, 2, \dots, K+1$ by the absolute continuity of the integral. Using Corollary~\ref{GridChooser2} we find a grid
$$
	\Gamma= \bigcup_{i=1}^K \{x_i\}\times [-1,1] \cup \bigcup_{j=1}^K[-1,1]\times \{y_j\}
$$
such that $x_j,y_j\in [-1+\tfrac{2}{K+1} (j -\tfrac{3}{4}), -1+\tfrac{2}{K+1} (j-\tfrac{1}{4})]$ with $f_{|\Gamma}$ continuous at each $(x_j,y_m)$. Moreover by a simple argument of averages we choose the ordinates so that for each $j,m=1,2,\dots K$ it holds that
$$
|D_1f|([-1,1]\times\{y_m\}) \leq 2(K+1)\int_{-1+\tfrac{2}{K+1}(m-1)}^{-1+\tfrac{2}{K+1}m} |Df_1|([-1,1]\times\{y\}) \, dy < 2(K+1)\epsilon^2
$$
and
$$
|D_2f|(\{x_j\}\times[-1,1]) \leq 2(K +1)\int_{-1+\tfrac{2}{K+1}(j-1)}^{-1+\tfrac{2}{K+1}j} |Df_2|(\{x\}\times[-1,1]) \, dx < 2(K+1)\epsilon^2.
$$

Denoting for simplicity $y_0 = -1$ and $y_{K+1}= 1$, from
$$
|D_2f|(\{x_j\}\times[-1,1]) = \sum_{m=0}^K|D_2f|(\{x_j\}\times[y_m, y_{m+1}])<2(K+1)\epsilon^2
$$
it follows that the number of intervals $[y_m, y_{m+1}]$ such that $|D_2f|(\{x_j\}\times[y_m, y_{m+1}])> \epsilon$ is at most $\left \lfloor{2( K+1)\epsilon}\right \rfloor$. That is to say that the number of rectangles in a column with a vertical side whose image has length greater than $\epsilon$ is at most $\left \lfloor{4(K+1)\epsilon}\right \rfloor$, where $\left \lfloor{a}\right \rfloor$ denotes the integer part of $a$.
The same holds for horizontal rows and, being $(K+1)$ the total number of such rows (or columns), we get
\begin{equation}\label{SmallOscillation}
\diam f_{|\partial R} < 4\epsilon \text{ for all but an $\left \lfloor{8(K+1)^2\epsilon}\right \rfloor$ number of rectangles $R$ of $\Gamma$}.
\end{equation}
On the remaining \emph{bad} rectangles $R$ of $\Gamma$ we can use the following estimate which is always true 
\begin{equation}\label{BigOscillation}
\diam f_{|\partial R} < 4.
\end{equation}
Note that the measure of each such rectangle is bounded by $\frac{16}{(K+1)^2}$ and so 
\begin{equation}\label{Theatre}
\text{the measure of the union of the\emph{ bad} rectangles is at most } C\epsilon.
\end{equation}

 \step{2}{Choice of guidelines}{guide}

Taking $\Gamma$, the good starting grid chosen in step~\ref{start}, we denote by
$$
	H_2 =  \left\lbrace y \in [-1,1]: \exists x_j : |D_2 f|(x_j,y)\geq\frac{\epsilon}{K+1} \right\rbrace.
$$
Similarly we denote
$$
	H_1 = \left\lbrace x \in [-1,1]: \exists y_m : |D_1 f|(x,y_m)\geq\frac{\epsilon}{K+1} \right\rbrace.
$$
Then the set $H=(H_1\times[-1,1]\cup[-1,1] \times H_2)\cap \Gamma$ contains the coordinates of all ``big'' jumps on the grid $\Gamma$. Since we chose the grid so that $f$ is $BV$ on all lines of $\Gamma$, we immediately see that $H$ is finite. For each $x\in H_1$ we choose an interval $(x-s_x, x+s_x)$ and for each $y\in H_2$ we choose an interval $(y-s_y, y+s_y)$, so that the numbers $s_x,s_y>0$ are so small that the segments $[x-s_x,x+s_x]\times\{y_m\}$, $m=0, 1, \dots K+1$ and $\{x_j\}\times[y-s_y, y+s_y]$, $j=0,1,\dots K+1$ satisfy the following properties
\begin{enumerate}
	\item[i)] each segment contains exactly one $(x,y) \in H$
	\item[ii)] the segments have pairwise positive distance from each other
	\item[iii)] the segments do not contain any vertex $(x_j,y_m)$ of $\Gamma$
	\item[iv)] the $\mathcal{H}^1$ measure of the union of the segments is bounded by $\frac{\epsilon}{(K+1)^2|D_{\tau}f|(\Gamma)}$.
\end{enumerate}
We call $E$ the finite union of these segments.

We now add more lines to the grid $\Gamma$ that will be useful in the sequel to obtain necessary estimates for the injectification. 
These extra lines are called horizontal and vertical \emph{guidelines}. 

Consider a rectangle $R = (x_j,x_{j+1})\times(y_m,y_{m+1})$ of $\Gamma$, then we cover $\partial R \setminus E$ with corresponding pairs of segments $J_{R,i}^-$ and $J_{R,i}^+$, where $J_{R,i}^-$ is on the left (resp. lower) side of $R$ and $J_{R,i}^+$ is on the right (resp. upper) side of $R$. We require that $\pi_i(J_{R,i}^{-})=\pi_i(J_{R,i}^{+})$, where $\pi_i$ is the projection in the direction perpendicular to $J_{R,i}^{\pm}$ and further that $|D_1f|(J_{R,i}^{\pm}) < 2\frac{\epsilon}{K+1}$ if $J_{R,i}^{\pm}$ are horizontal and $|D_2f|(J_{R,i}^{\pm}) < 2\frac{\epsilon}{K+1}$ if $J_{R,i}^{\pm}$ are vertical. 
Now, for each pair of horizontal $J_{R,i}^{\pm}$, we choose a vertical line $\{t_{R,i}\}\times[-1,1]$ intersecting both $J_{R,i}^{\pm}$. Similarly for each pair of vertical $J_{R,i}^{\pm}$ we choose a horizontal line $[-1,1]\times\{t_{R,i}\}$ intersecting $J_{R,i}^{\pm}$ such that $f$ is continuous at the mutual intersections of these lines and at their intersections with $\Gamma$. Since, in the terminology of the proof of Lemma~\ref{GridChooser}, the lines in $\Gamma$ have ordinates chosen from the set $(A \setminus \tilde{A})$ and $(B \setminus \tilde{B})$ respectively, Corollary~\ref{GridChooser2}, gives that almost any selection for the values $t_{R,i}$ is acceptable. 

Therefore we can always choose $t_{R,i}$ so that the respective line is almost minimizing the total variation of $f$ in the strip corresponding to $J_{R,i}^{\pm}$. More precisely, if $J_{R,i}^{\pm}$ were vertical in the rectangle $R$, then we can choose $t_{R,i}$ so that
\begin{equation}\label{GlDefH}
|D_1 f|((x_j,x_{j+1})\times\{t_{R,i}\})\leq  |D_1 f|((x_j,x_{j+1})\times\{t\})+\frac{\epsilon}{K+1} 
\end{equation}
for almost every $(x_j,t)\in J_{R,i}^-$. In this case the line $[-1,1]\times\{t_{R,i}\}$ is called \emph{horizontal guideline} for $R$.

Similarly, if $J_{R,i}^{\pm}$ were horizontal in $R$, then we consider $t_{R,i}$ so that
\begin{equation}\label{GlDefV}
|D_2 f|(\{t_{R,i}\}\times(y_m,y_{m+1}))\leq  |D_2 f|(\{t\}\times (y_m,y_{m+1}))+\frac{\epsilon}{K+1} 
\end{equation}
for almost every $(t,y_m)\in J_{R,i}^-$ and the line $\{t_{R,i}\}\times[-1,1]$ is called \emph{vertical guideline} for $R$.

We denote by  $\tilde \Gamma$ the grid formed by adding the horizontal and vertical guidelines of any rectangle to $\Gamma$. Observe, that $\tilde \Gamma$ is still a good starting grid in the sense of Definition \ref{good starting grid}. Then let $\gamma$ be the geometrical representation of $f$ on the grid $\tilde \Gamma$. We recall that, by definition, $\gamma$ coincides with $f$ at every intersection of two lines of $\tilde{\Gamma}$.

\step{3}{Choice of an arrival grid}{arrive}

Let $\tilde{\Gamma}$ and $\gamma$, be the good starting grid and the corresponding geometric representative of $f$ on $\tilde{\Gamma}$ defined in step~\ref{guide}. Then Lemma~\ref{ArrivalGrid} provides a good arrival grid associated with $\tilde{\Gamma}$, $\gamma$ with sidelength $\eta = \frac{\epsilon}{K+1}$ in the sense of Definition \ref{good arrival grid}.
 
\step{4}{Injectification and linearization of $\gamma$}{inject}

In this step we employ the good arrival grid provided by step~\ref{arrive} to linearize the injectification of $f |_{\tilde \Gamma}$. This concept was pioneered in \cite{PP}, in the present manuscript we adapt it to the $BV$ setting. 

By definition of good arrival grid $\mathcal{G}$, the set $P:=\gamma^{-1}(\mathcal G)\cap \tilde{\Gamma}$ is finite and does not contain any vertices of $\mathcal{G}$ or of $\tilde{\Gamma}$. 
Further for every point $(x,y)\in P$ it holds that the derivative of $\gamma$ tangential to $\tilde{\Gamma}$ (we denote it as $D_{\tau}\gamma$) has non-zero component perpendicular to the side of $\mathcal{G}$ containing $\gamma(x,y)$ (the existence of $D_{\tau}\gamma$ at all points of $P$ is a requirement of the good arrival grid, see Definition \ref{good arrival grid}). Therefore there exists a smallest such perpendicular component $v>0$.

For each point $a\in P$ we have some $d_{a}>0$ such that when $|(x,y)-a|< d_{a}$ then
$$
	\gamma(x,y)-\gamma(a) - D_{\tau}\gamma(a)[(x,y)-a]< \frac{v}{3}|(x,y) - a|.
$$
As a consequence, the images through $\gamma$ of the endpoints of the segments $B(a, d_{a})\cap\tilde{\Gamma}$ have distance at least $\frac{vd}{2}$ from $\gamma(a)$, where we called $d = \min_{a\in P}d_{a} > 0$. Without loss of generality, we can assume that $B(a, d_{a})\cap\tilde{\Gamma}$ is a finite collection of segments which are pairwise disjoint and (up to decrease the values of $d_{a}$) do not contain any vertex of $\tilde{\Gamma}$. Moreover, since $\gamma$ is continuous and $P$ is finite, up to further decrease the values of $d_a$ (for example by a factor $2$), we deduce that if $c$ is an endpoint of any segment of $B(a, d_{a})\cap\tilde{\Gamma}$ then $|\gamma(c) - \gamma(a)| \geq \frac{vd}{2}$  for all $a \in P$.

On the other hand, being $\gamma \big(\tilde{\Gamma} \setminus \bigcup_A B(a, d_{a}) \big)$ a closed set, it follows that there exists a $\sigma_0>0$ such that
\begin{equation}\label{Marathon}
	\dist\bigg(\gamma \Big(\tilde{\Gamma} \setminus \bigcup_AB\big(a,d_{a}\big)\Big), \mathcal{G}\bigg) \geq 2\sigma_0.
\end{equation}
This immediately implies that for any $0<\sigma\leq \sigma_0$ and any $\tilde{\gamma}_{\sigma}$, continuous injectification of $\gamma$ with $\|\tilde{\gamma}_{\sigma} - \gamma\|_{L^\infty(\tilde{\Gamma})} \leq \sigma$, it holds that $\tilde{\gamma}_{\sigma}^{-1}(\mathcal{G})\subset \bigcup_A B\big(a,d_{a}\big) \cap \tilde{\Gamma}$.

	Let 
	\begin{equation}
		\delta'=\min\{|\mathbf{a}-\mathbf{b}|\colon \mathbf{a}\in \gamma(P),\, \mathbf{b} \textrm{ vertex of }\mathcal{G}\}     
	\end{equation}
	and  
	\begin{equation}
		\delta''=\min\{|\mathbf{a}-\mathbf{b}|\colon\,  \mathbf{a},\mathbf{b}\in \gamma(P)\}
	\end{equation}
	Finally we set 
	\begin{equation}
		\delta=\min\{\delta',\,\delta'', \sigma_0, \tfrac{1}{100}\}.
	\end{equation}
	Notice that $\delta$ is positive due to the properties of good arrival grid.
	Let 
	\begin{equation}\label{SigmaDef}
		0<\sigma\leq \frac{\epsilon^2\delta }{ 12(K+1)}
	\end{equation}
	Now we use the NCBV condition to  obtain an injective function $\tilde{\gamma}_{\sigma}:\tilde\Gamma\to Q(0,1)$ such that $\norm{\tilde{\gamma}_{\sigma}-\gamma}_{L^\infty(\tilde{\Gamma})}<\sigma$.

	We adjust the map $\tilde{\gamma}_{\sigma}$ as follows. For each $a\in P=\gamma^{-1}(\mathcal G)\cap \tilde{\Gamma}$ we find the first and last point (i.e. the points furthest away from $a$) on the segment $B(a, d_{a})\cap\tilde{\Gamma}$ (call them $a^-$ and $a^+$ respectively) such that $\tilde{\gamma}_{\sigma}(a^{\pm}) \in B(\gamma(a), 3\sigma)$. Notice that \eqref{Marathon} and the choice of $\sigma$ imply that  
	\[ \gamma \Big(\tilde{\Gamma} \setminus \bigcup_AB\big(a,d_{a}\big)\Big) \cap \Big( \bigcup_{a \in P} B(\gamma(a), 3\sigma) \Big) = \emptyset,  \] 
	hence we are sure that $\tilde{\gamma}_\sigma$ intersects $B(\mathbf{a}, 3\sigma)$ only on the segments  $B(a, d_{a})\cap\tilde{\Gamma}$, $a\in P$, $\gamma(a) = \mathbf{a}$.
	 We define
	$$
	\tilde{\tilde{\gamma}}_{\sigma}(t) = \frac{|t - a^+|}{|a^+ - a^-|}\tilde{\gamma}_{\sigma}(a^-) +\frac{|t - a^-|}{|a^+ - a^-|} \tilde{\gamma}_{\sigma}(a^+) \qquad \mbox{for all $t \in [a^-a^+]$ and $a \in P$}
	$$
	then we set
	$$
	\tilde{\tilde{\gamma}}_{\sigma}(t) ={\tilde{\gamma}}_{\sigma}(t) \qquad \mbox{otherwise on $\tilde{\Gamma} \setminus \bigcup_A [a^- a^+]$.}
	$$
	By construction, it follows that $\tilde{\tilde{\gamma}}_{\sigma}(t)$ is again injective and $\|\tilde{\tilde{\gamma}}_{\sigma} -\gamma\|_{L^\infty(\tilde{\Gamma})}\leq 7\sigma$. Since $\tilde{\tilde{\gamma}}_{\sigma}(a^-)$ and $\tilde{\tilde{\gamma}}_{\sigma}(a^+)$ must be separated by $\mathcal{G}$ there is exactly one point $\tilde{a}$ in each $[a^-a^+]$ which is mapped onto $\mathcal{G} \cap B(\mathbf{a}, 3\sigma)$.

	At this stage we use $\tilde{\tilde{\gamma}}_{\sigma}$ and the arrival grid $\mathcal{G}$ to define a piecewise linear map from $\tilde{\Gamma}$ to $\er^2$. We will call this map $\gamma_{\sigma}$. 
	We start by specifying the image, $\gamma_{\sigma}(\tilde{\Gamma})$. 
	For each line $[-1,1]\times\{y\}$ ( resp. $\{x\} \times[-1,1]$) contained in $\tilde{\Gamma}$ we have a finite number of points $\tilde{a}$ such that $\tilde{\tilde{\gamma}}_{\sigma}(\tilde{a}) \in \mathcal{G}$. 
	For each pair of adjacent $\tilde{a}_1, \tilde{a}_2$ for which $\tilde{\tilde{\gamma}}_{\sigma}(\tilde{a}_1),\tilde{\tilde{\gamma}}_{\sigma}(\tilde{a}_2)$ lie on two distinct sides of a rectangle in $\mathcal{G}$ we define the segment $S_{a_1,a_2} = [\tilde{\tilde{\gamma}}_{\sigma}(\tilde{a}_1),\tilde{\tilde{\gamma}}_{\sigma}(\tilde{a}_2)]$ where $a_1$ and $a_2$ are the unique points in $P$ for which $\tilde{a}_i \in B(a_i, d_{a_i})$.
	
	Let us now consider a pair of adjacent $\tilde{a}_1$ and $\tilde{a}_2$ for which $\tilde{\tilde{\gamma}}_{\sigma}(\tilde{a}_1),\tilde{\tilde{\gamma}}_{\sigma}(\tilde{a}_2)$ lie on the same side of a rectangle in $\mathcal{G}$. 
	Firstly notice that for any such pair $\tilde{a}_1$ and $\tilde{a}_2$ there exists an $0<\xi_{a_1,a_2}$ so small that the generalized segments (see Definition \ref{generalized segments}) with $\xi = \xi_{a_1,a_2}$ intersect only those previously defined straight segments $S_{a_3,a_4}$ for which $\tilde{\tilde{\gamma}}_{\sigma}([\tilde{a}_1,\tilde{a}_2])$ was already intersecting $\tilde{\tilde{\gamma}}_{\sigma}([\tilde{a}_3,\tilde{a}_4])$. We define $\xi_0 = \tfrac{1}{2}\min\xi_{a_1,a_2}$ where the minimum is taken over all adjacent pairs $a_1,a_2 \in P$ with images lying on a common side of $\mathcal{G}$.
	 Then let  $\xi = \min\{\xi_0, \epsilon\}$ and define $S_{a_1,a_2}$ as the generalised segment from $\tilde{\tilde{\gamma}}_{\sigma}(\tilde{a}_1)$ to $\tilde{\tilde{\gamma}}_{\sigma}(\tilde{a}_2)$ with the chosen $\xi$.	

	It is very easy to check that any pair of the thusly defined paths $S_{a_1,a_2}$ and $S_{a_3,a_4}$ ,for $a_1,a_2$ and $a_3,a_4$ pairs of adjacent points of $P$, intersect each other if and only if $\tilde{\tilde{\gamma}}_{\sigma}([\tilde{a}_1\tilde{a}_2])$ intersects $\tilde{\tilde{\gamma}}_{\sigma}([\tilde{a}_3\tilde{a}_4])$. Further, any such pair can have at most one intersection for both sets of paths. 
	
	Now we are in a position to define the map $\gamma_{\sigma}$ on $\tilde{\Gamma}$. We define $\gamma_{\sigma}(a) = \tilde{\tilde{\gamma}}_{\sigma}(\tilde{a})$ for all $a\in P$ and for all the corresponding $\tilde{a}$. 
	Further, for every vertex $V$, intersection of a vertical and horizontal line of $\tilde{\Gamma}$, there exists exactly two pairs of adjacent $a_1,a_2\in P$ (on the horizontal line) and $a_3,a_4 \in P$ (on the vertical line) such that $V = [a_1a_2]\cap[a_3a_4]$ and no other points $a \in P$ are included in $ [a_1a_2]\cup[a_3a_4]$. 
	
	 Then, by construction, there exists exactly one point of intersection call it $\mathbf{V}$ in the set $S_{a_1,a_2}\cap S_{a_3,a_4}$ and we define $\gamma_{\sigma}(V) = \mathbf{V}$. Thus we have separated the grid $\tilde{\Gamma}$ into simple segments lying between adjacent intersections with $\mathcal{G}$, intersecting segments of $\tilde{\Gamma}$ or a combination of the two. In each case there is a clear correspondence between the endpoints of remaining segments in $\tilde{\Gamma}$ and (parts of the possibly generalized) segments defined in the previous paragraph. We parametrize these segments (or possibly paths consisting of 2 segments) by arc length from the corresponding segments in $\tilde{\Gamma}$.
	 
	 Thus we obtain a continuous injective piecewise linear mapping $\gamma_{\sigma}:\tilde{\Gamma}\to Q(0,1)$ satisfying $\|\gamma - \gamma_{\sigma}\|_{L^\infty(\tilde{\Gamma})}\leq 4\eta \leq \eps 4(K+1)^{-1}$.

\step{5}{Basic length estimates on $\tilde{\Gamma}$}{basic}

	We claim that for any pair of points $a,b \in [ab]\subset \tilde{\Gamma}$ we have that
	\begin{equation}\label{BasicGen}
		|D_{\tau}\gamma_{\sigma}|([ab]) \leq (1+\xi)(1+\epsilon) \left(|D_{\tau}\gamma|([ab]) + 8\frac{\epsilon}{K+1} \right).
	\end{equation}
	Indeed it holds that $|\gamma(a) - \gamma_{\sigma}(a)|\leq 3\sigma \leq \tfrac{1}{2}\epsilon\delta$ for each $a\in P$. For any pair $a, a' \in P$ adjacent on a segment of $\tilde{\Gamma}$ we have that either $\gamma(a) = \gamma(a')$ or $|\gamma(a)- \gamma(a')| \geq \delta$ (this is immediate from the definition of $\delta$). In the first case we know that the length of the curve given by $\gamma$ on $[aa']$ is at least $100\sigma$ by \eqref{Marathon}, $\epsilon<\tfrac{1}{100}$, and $\sigma\leq \epsilon^2 \delta \leq \epsilon \sigma_0$. On the other hand we have $|\gamma_{\sigma}(a)- \gamma_{\sigma}(a')| < 6\sigma$ and so the length of $S_{aa'}$ is at most $(1+\xi)6\sigma \leq 12\sigma < 100\sigma$. Therefore on such segments we in fact have that $|D_{\tau}\gamma_{\sigma}|([aa']) < |D_{\tau}\gamma|([aa'])$.
	
	In the second case we have that $|\gamma(a)- \gamma(a')| \geq \delta$ and, by \eqref{SigmaDef}, that $|\gamma_{\sigma}(a) - \gamma(a)| < 3\sigma\leq  \epsilon\delta$. We can estimate by the triangle inequality that
	$$
		\begin{aligned}
			|\gamma_{\sigma}(a) - \gamma_{\sigma}(a')|
			&\leq |\gamma_{\sigma}(a) - \gamma(a)|+|\gamma(a) - \gamma(a')|+|\gamma(a') - \gamma_{\sigma}(a')|\\
			& \leq |\gamma(a) - \gamma(a')| + 6\sigma\\
			& \leq (1+\epsilon)|\gamma(a) - \gamma(a')|.
		\end{aligned}
	$$
	Now, because the length of the generalized segment between $\mathbf{a}$ and $\mathbf{b}$ with parameter $\xi$ has length bounded by $(1+\xi)|\mathbf{a} - \mathbf{b}|$, we get that
	\begin{equation}\label{BasicA}
		|D_{\tau}\gamma_{\sigma}|([aa']) \leq (1+\xi)(1+\epsilon)|D_{\tau}\gamma|([aa'])
	\end{equation}
	for any $a,a'\in P$ adjacent on a segment in $\tilde{\Gamma}$ and hence also immediately for any $a,a'\in P$ lying on a segment of $\tilde{\Gamma}$ but not necessarilly adjacent.

	The argument for a general pair $a,b$ both lying on a single segment of $\tilde{\Gamma}$ is as follows. Assume that there is at least one $a'$ lying on $[ab]$. Then the estimate \eqref{BasicA} takes care of the maximal segment lying  between two points of $P$ fully contained in $[ab]$ leaving only two end bits. Observe, that the image curve of each end bit is contained in a rectangle $R$ in $\mathcal{G}$ whose diameter is bounded by $4\frac{\epsilon}{K+1}$. Then the generalized segment image of the end bits must also be bounded by $(1+\xi)4\frac{\epsilon}{K+1}$. There are 2 such end bits, hence the `$8\frac{\epsilon}{K+1}$'-term in the estimate in \eqref{BasicGen}.

\step{6}{Upper bounds on geodesics in $\gamma_{\sigma}(R)$}{MeasureMe}

	Having chosen an injective $\gamma_{\sigma}$ on $\tilde{\Gamma}$ we intend to use the minimal extension theorem, Theorem~\ref{MinExt} of \cite{PR1} on rectangles $R$ of $\Gamma$ with boundary datum given by $\gamma_{\sigma}$ restricted to $\Gamma$. In order to bound the total variation of the extension given by the theorem, it is necessary to bound the length of geodesics between the images of opposing points. We are able to do this (at least for a large number of points) thanks to our choice of the guidelines.
	
	Let $R = [x_{j-1},x_j]\times[y_{m-1},y_m]$ be a rectangle of $\Gamma$. We consider points on opposing vertical sides but the argument for horizontal sides is analogous. Therefore let $(x_{j-1},y) \in J_{R,i}^-$ and $(x_j,y)$ be its opposing point in $J_{R,i}^+$. We recall that $[-1,1]\times\{t_{R,i}\}\subset\tilde{\Gamma}$ is a guideline for $R$. The segments $J_{R,i}^{\pm}$ were chosen so that $|D_2f|(J_{R,i}^{\pm}) < 2\frac{\epsilon}{K+1}$ and so using \eqref{BasicGen} we get that
	$$
		\begin{aligned}
			|D_2\gamma_{\sigma}([(x_{j-1},y),(x_{j-1},t_{R,i})])| 
			&\leq (1+\xi)(1+\epsilon)\left(|D_2\gamma|([(x_{j-1},y),(x_{j-1},t_{R,i})])+8\frac{\epsilon}{K+1}\right)\\
			&\leq (1+\xi)(1+\epsilon)\frac{10\epsilon}{K+1}.
		\end{aligned}
	$$
	Similarly $|D_2\gamma_{\sigma}([(x_{j},y),(x_{j},t_{R,i})])| \leq (1+\xi)(1+\epsilon)\frac{10\epsilon}{K+1}$. But by \eqref{BasicA} we estimate
	$$
		\begin{aligned}
			|D_1\gamma_{\sigma}([(x_{j-1},t_{R,i}),(x_{j},t_{R,i})])| 
			&\leq (1+\xi)(1+\epsilon)|D_1\gamma|([(x_{j-1},t_{R,i}),(x_{j},t_{R,i})])\\
			&= (1+\xi)(1+\epsilon)|D_1f|([(x_{j-1},t_{R,i}),(x_{j},t_{R,i})]).
		\end{aligned}
	$$
	Thanks to \eqref{GlDefH}, for almost all $y$ such that $(x_{j-1},y) \in J_{R,i}^-$ and $(x_{j},y)\in J_{R,i}^+$ follows that 
	$$
		|D_1f|([(x_{j-1},t_{R,i}),(x_{j},t_{R,i})]) \leq |D_1f|([(x_{j-1},y),(x_{j},y)]) + \frac{\epsilon}{K+1}.
	$$
	Gathering together these estimates, we deduce that for almost all pairs $(x_{j-1},y), (x_{j},y)$ in $J_{R,i}^{\pm}$ there is a path in $\tilde{\Gamma}$ (call it $Z$) connecting $(x_{j-1},y)$ with $(x_{j},y)$ which has length
	$$
		|D_{\tau}\gamma_{\sigma}|(Z) \leq (1+\xi)(1+\epsilon)\Big(|D_1f|([(x_{j-1},y),(x_{j},y)]) +21\frac{\epsilon}{K+1}\Big).
	$$

Observe, that the curve $\gamma_{\sigma}(Z)$ is contained in the polygon $\mathcal{P}$ identified by the piecewise linear Jordan curve $\gamma_{\sigma}(\partial R)$. In particular we have $\partial \mathcal{P} = \gamma_{\sigma}(\partial R)$.
Moreover, $\gamma_\sigma (x_{j-1},y)$ and $\gamma_\sigma (x_{j},y)$ are the endpoints of $\gamma_{\sigma}(Z)$ and are contained in $\gamma_{\sigma}(\partial R)$. As a consequence, the length of the shortest curve contained in $\mathcal{P}$ and connecting the same endpoints is smaller than the length of $\gamma_{\sigma}(Z)$. 
As in Theorem~\ref{MinExt}, given $a,b \in \partial \mathcal{P}$ we denote by $d_{\mathcal{P}}(a,b)$ the length of the geodesic laying in $\mathcal{P}$ and connecting $a$ and $b$. Then the above comment can be rephrased as 
	
	\begin{equation}\label{DobryOdhad}
		d_{\mathcal{P}}(\gamma_{\sigma}(x_{j-1},y), \gamma_{\sigma}(x_{j},y)) \leq 	(1+\xi)(1+\epsilon)\Big(|D_1f|([(x_{j-1},y),(x_{j},y)]) + 21\frac{\epsilon}{K+1}\Big)
	\end{equation}
	for almost every $[(x_{j-1},y),(x_{j},y)]$ with endpoints in $J_{R,i}^{\pm}$. For vertically opposing points we have the same.

	We conclude this step with an estimate for $d_{\mathcal{P}}(\gamma_{\sigma}(x_{j-1},y),\gamma_{\sigma}(x_{j},y))$ for points $(x_{j-1},y), (x_{j},y) \in E$. In this case we simply estimate that
	\begin{equation}\label{DebilniOdhad}
	d_{\mathcal{P}}(\gamma_{\sigma}(x_{j-1},y),\gamma_{\sigma}(x_{j},y)) \leq |D_{\tau}f|(\partial R) \leq |D_{\tau}f|(\Gamma).    \quad 
	\end{equation}
	and similarly for $(x,y_{m-1}), (x,y_m) \in E$.

\step{7}{Extension and variation estimates}{Final}

For each $[x_{j-1},x_{j}]\times[y_{m-1},y_m] = R$ of $\Gamma$ we have a polygon $\Omega_R$ whose boundary is paramterized by the piecewise linear map $\gamma_\sigma |_{\partial R}$. Then, by Theorem~\ref{MinExt} applied to $\gamma_{\sigma}|_{\partial R}$, we have a homeomorphism $h_R:R\to \Omega_R$ extending $\gamma_\sigma$ such that 
\begin{equation}\label{HomeoVarEst}
	\begin{aligned}
		&\int_{y_{m-1}}^{y_m}|D_1 h_R|((x_{j-1},x_{j})\times  \{t\})\,dt+\int_{x_{j-1}}^{x_{j}}|D_2 h_R|( \{t\} \times (y_{m-1}, y_m))\,dt\\
		\leq &\int_{y_{m-1}}^{y_m} d_{\mathcal{P}}(\gamma_{\sigma}(x_{j-1},t),\gamma_{\sigma}(x_{j},t))\,dt +\int_{x_{j-1}}^{x_{j}} d_{\mathcal{P}}(\gamma_{\sigma}(t,y_{m-1}),\gamma_{\sigma}(t,y_m))\,dt\\
		&\quad+\frac{\epsilon}{K+1}.
	\end{aligned}
\end{equation}

We  define $h:Q(0,1)\to Q(0,1)$ by setting $h|_R=h_R$ for all $R$ in the grid 
$\Gamma.$ By the choice of the boundary values of $h_R$ it is clear that $h$ is 
a homeomorphism.

Let us estimate $\|h-f\|_{L^1(Q(0,1))}$. For every $(x_j,y_m)$ vertex of $\Gamma$ it holds that $h(x_j,y_m)$ lies in the same rectangle of $\G$ as does $f(x_j,y_m)$ (we take $f(x_j,y_m)$ as the unique value such that $f$ is continuous at $(x_j,y_m)$ w.r.t. $\Gamma$). The diameter of this rectangle is at most $4\epsilon K^{-1}$. Therefore $|h(x_j,y_m) - f(x_j, y_m)|<C\epsilon$. We have that $\diam h(R) \leq \diam f(R) + C\epsilon K^{-1}$ because $h(\partial R) = \gamma_{\sigma}(\partial R) \subset \gamma(\partial R) + B(0,4\epsilon K^{-1})$. 
Combining the above facts, for any rectangle on which $\diam f(\partial R) \leq \epsilon$ we have $\|f-h\|_{L^\infty(R)}<C\epsilon$. On the other hand, there are at most $\left \lfloor{8(K+1)^2\epsilon}\right \rfloor$ \emph{bad} rectangles where $\diam f(\partial R)>\epsilon$ (see \eqref{SmallOscillation} and \eqref{BigOscillation}). From \eqref{BigOscillation} we can still estimate $\int_R|f-h| \leq 4\mathcal{L}^2(R)$ whenever $R$ is one such \emph{bad} rectangle. By calling then $\Sigma  = \bigcup_{R \, \emph{bad}} R$ and using \eqref{Theatre}, we deduce
$$
	\int_{\Sigma}|f-h| \leq 4\mathcal{L}^2(\Sigma) \leq C\epsilon.
$$
 Therefore
 $$
 \|h-f\|_{L^1(Q(0,1))} \leq  \sum_{R \in Q(0,1) \setminus \Sigma} \|h-f\|_{L^\infty(R)} + \|h-f\|_{L^1(\Sigma)} \leq 2C\epsilon.
 $$

Recalling the notation introduced in Step 2, we have the following estimate for $|D_1 h|(Q(0,1))$ ($\abs{D_2 h}$ is estimated analogously).
\begin{equation}
	\begin{aligned}
  		\sum_{j,m=1}^{K+1}&\int_{y_{m-1}}^{y_m}|D_1 h|((x_{j-1},x_j)\times\{t\}) \, dt\\ =&\sum_{j,m=1}^{K+1}\int_{\{t\in[y_{m-1}, y_{m}]\colon (x_j,t) \notin E \}}|D_1 h|((x_{j-1},x_j)\times\{t\}) \, dt\\
  		&+ \sum_{j,m=1}^{K+1}\int_{\{t\in[y_{m-1}, y_{m}]\colon (x_j,t)\in E\}}|D_1 h|((x_{j-1},x_j)\times\{t\}) \, dt\\
  		=:&\,I\,+\,II.
	\end{aligned}
\end{equation}
We start with $II$. For each $[x_{j-1},x_j]\times[y_{m-1},y_m] = R$ and for each $t \in [y_{m-1}, y_{m}]$ such that $(x_{j-1},t)$ and $(x_j , t)$ belong to $E$ we estimate by \eqref{DebilniOdhad}. But then (recall the definition of $E$ especially that $\mathcal{H}^1(E) \leq \frac{\epsilon}{(K+1)^2|D_{\tau}f|(\Gamma)}$) we have that
$$
	\int_{\{t\in[y_{m-1}, y_{m}]\colon (x_j,t)\in E_R \}}|D_1 h|((x_{j-1},x_j)\times\{t\}) \, dt \leq \frac{\epsilon}{(K+1)^2} 
$$
and summing over $j,m=1,\ldots, K+1$ we get 
\begin{equation}\label{BSolved}
	II \leq \epsilon.
\end{equation}

For $I$ we use \eqref{HomeoVarEst} and \eqref{DobryOdhad} (recall $\xi\leq \epsilon$) to obtain
$$
	\begin{aligned}
		I&\leq \sum_{j,m=1}^{K+1} \int_{y_{m-1}}^{y_m} d_{\mathcal{P}}(\gamma_{\sigma}(x_{j-1},t), \gamma_{\sigma}(x_j,t))+\frac{\epsilon}{K+1} \,dt \\ 
		&\leq \sum_{j,m=1}^{K+1}\int_{y_{m-1}}^{y_m}(1+\epsilon)^2\Big(|D_1f|([(x_{j-1},t)(x_j,t)]) + 21\frac{\epsilon}{K+1}\Big) +\frac{\epsilon}{K+1} \,dt \\
		&\leq \sum_{j=1}^{K+1} \int_{-1}^{1}(1+\epsilon)^2\Big(|D_1f|([(x_{j-1},t)(x_j,t)]) + 21\frac{\epsilon}{K+1}\Big) +\frac{\epsilon}{K+1} \,dt \\
		& \leq (1+\epsilon)^2\Big(|D_1 f|(Q(0,1))+42\epsilon\Big)+2\epsilon\\
		&\leq  |D_1 f|(Q(0,1) + 3\epsilon|D_1 f|(Q(0,1) +86\epsilon.
	\end{aligned}
$$
Thus we have shown that for every $\epsilon>0$ there exists a homeomorphism $h$ such that $\|f-h\|_{L^1(Q(0,1))} < C\epsilon$ and $|D_1 h|(Q(0,1))\leq |D_1 f|(Q(0,1))+ C\epsilon$ and similarly $|D_2 h|(Q(0,1))\leq |D_2 f|(Q(0,1))+ C\epsilon$.

By compactness we may form a sequence $h_j$ of such homeomorphisms that weak-* converges to $f.$ From Reshetnyak's lower semicontinuity theorem \cite[2.38]{AFP}  
and the construction of $h_j$ we see that 
\begin{equation*}
 \begin{aligned}
\limsup_{j\rightarrow \infty}& |D_1 h_j|(Q(0,1))+|D_2 h_j|(Q(0,1))\leq |D_1 f|(Q(0,1))+|D_2 f|(Q(0,1)) \\&\leq\liminf_{j\rightarrow \infty} |D_1 h_j|(Q(0,1))+|D_2 h_j|(Q(0,1)).
  \end{aligned}
\end{equation*}
That is, $h_j$ converges to $f$ in the sense of \eqref{Manhattan strict convergence}. 
\end{proof}



\end{document}